\documentclass[11pt]{amsproc}
\usepackage{amsfonts}
\usepackage{amsmath,amstext,amsbsy,amssymb}

\newtheorem{theorem}{Theorem}[section]
\newtheorem{lemma}[theorem]{Lemma}

\newtheorem{proposition}[theorem]{Proposition}
\newtheorem{corollary}[theorem]{Corollary}

\theoremstyle{definition}
\newtheorem{definition}[theorem]{Definition}

\theoremstyle{remark}
\newtheorem{remark}[theorem]{Remark}

\numberwithin{equation}{section}

\newcommand{\la}{\lambda}
\newcommand{\al}{\alpha}
\begin{document}

\title{Applications of Laplace-Beltrami operator for Jack polynomials}
\author{Wuxing Cai}
\address{School of Sciences,
South China University of Technology, Guangzhou 510640, China}
\email{caiwx@scut.edu.cn}
\author{Naihuan Jing$^*$}
\address{Department of Mathematics,
   North Carolina State University,
   Raleigh, NC 27695-8205, USA}
\email{jing@math.ncsu.edu}
\thanks{*Corresponding author}
\thanks{Jing gratefully acknowledges the support from
NSFC's Overseas Young Collaborative Grant (10801094).}
\keywords{Laplace-Beltrami operators, Symmetric functions, Jack
polynomials, Virasoro algebra} \subjclass[2000]{Primary: 05E05;
Secondary: 17B69, 05E10}

\begin{abstract}
We use a new method to study the Laplace-Beltrami type operator on the Fock space of
symmetric functions, and as an example of our explicit computation we show that the Jack symmetric functions
are the only family of eigenvectors of the differential operator.
As applications of this explicit method we find a combinatorial formula for Jack symmetric functions
and  the Littlewood-Richardson coefficients in the Jack case.
As further applications, we obtain a new determinantal formula for Jack symmetric functions.
We also obtained a generalized raising operator formula for Jack symmetric functions, 
and a formula for the explicit action
of Virasoro operators. Special cases of our formulas imply
Mimachi-Yamada's result on Jack symmetric functions of rectangular shapes, as well as
the explicit formula for Jack functions of two rows or two columns.
\end{abstract}
 \maketitle
\section{Introduction}
For $\alpha\in\mathbb C$ the generalized Laplace-Beltrami operator
\begin{equation}
L(\alpha)=\frac{\alpha}2\sum_{i=1}^n(x_i\frac{\partial}{\partial
x_i})^2+\frac1{2}\sum_{i<j}^n\frac{x_i+x_j}{x_i-x_j}(x_i\frac{\partial}{\partial
x_i}-x_j\frac{\partial}{\partial x_j})
\end{equation}
is the Hamiltonian for the Calogero-Sutherland-Moses model.
Macdonald showed that the eigenstates are the Jack symmetric
polynomials \cite{Ja} in variables $x_1, x_2$, $\cdots$, $x_n$. When
$\alpha=1, 2, 1/2$ the Laplace-Beltrami operator is the radial
Laplace operator for the symmetric spaces over the complex, real and
quaternion fields. Like Schur functions, the Jack functions are also
polynomials in power sum symmetric functions
$p_k=\sum_{i=1}^nx_i^k$, $k=1, \cdots, n$. Moreover it is a
fundamental fact that these symmetric functions enjoy the stability
property that $P_{\lambda}(x_1, \cdots, x_n;\alpha)$ is the same
polynomial as long as $n>|\lambda|$, i.e., they are actually
polynomials in variables $p_k=\sum_{i=1}^{\infty}x_i^k$. It is
advantageous to view $P_{\lambda}(x; \alpha)$ as an element in
$\mathbb Q(\alpha)[p_1, p_2, \cdots]$.

It was shown by Soko \cite{So} that
\begin{equation}
L(\alpha)(m_{\la})=\left[\sum_{i=1}^{l(\la)}\left(\frac{\alpha}2\la_i^2+
\frac12(n+1-2i)\la_i\right)\right]m_{\la}+\sum_{\mu<\la}c_{\la\mu}m_{\mu}.
\end{equation}
The action of the finite Laplace operator can be used to study Jack symmetric functions,
for instance a determinant formula \cite{LLM} is a direct consequence.
It is clear that a direct generalization of this finite Laplace
operator will not work as the eigenvalue does not make sense when
$n\longrightarrow \infty$.

In this paper we consider a Laplace-Beltrami like operator in
infinitely many variables which has all the favorite properties
enjoyed by the finite Laplace-Beltrami operator, moreover the
eigenvalues of the new operator can be used to distinguish
eigenstates. Our starting point is the observation that the space
$\Lambda=\mathbb C[p_1, p_2, p_3\ldots]$ is actually the Fock space
for the infinite dimensional Heisenberg algebra generated by $h_n$
with relations
\begin{align*}
[h_m, h_n]=m\alpha\delta_{m, -n}Id.
\end{align*}
If we identify $h_{-n}$ with $p_n$, the ring of symmetric functions is isomorphic to
the Fock space of the Heisenberg Lie algebra, which is the canonical irreducible representation of the
latter. Under this identification the
generating function of the generalized homogeneous symmetric
functions $q_n$ is exactly half of the vertex operator \cite{J2,
CJ}.

There is an indirect method to show that the operator diagonalize
Jack functions. For our later purpose we still
provide a direct method to show that $P_{\lambda}(\alpha; x)$ are eigenfunctions of the
following graded differential operator
$$\sum_{i,j\geq
1}ij\alpha^2 p_{i+j}\frac{\partial^2}{\partial p_{i}\partial
p_{j}}+\sum_{i,j\geq 1}(i+j)\alpha
p_{i}p_{j}\frac{\partial}{\partial
p_{i+j}}+\alpha(\alpha-1)\sum_{i\geq 1}i^2
p_{i}\frac{\partial}{\partial p_{i}}
$$
and the eigenvalues can distinguish Jack functions.

Using our new method, we derive an explicit action of $D(\alpha)$ on the basis of
generalized homogeneous functions $q_{\la}(\alpha)$ or
monomial symmetric functions $m_{\lambda}$, which establish a priori
the triangularity of the transition matrix between these bases and
Jack polynomials in one step. We then use the differential operator
to give

(i) an explicit iteration formula for the coefficients of
$Q_{\lambda}$ in terms of $q_{\lambda}$;

(ii) a combinatorial formula for Littlewood-Richardson coefficient and
provide a formula for Stanley's conjecture for Jack polynomials;

(iii) reformulation of Stanley
formula for two columns and Jing-J\'ozefiak formula for two rows.

It is well-known that Jack polynomials span a representation of
Virasoro algebra and the extremal vectors or singular vectors are
the Jack polynomials of rectangular shapes. This result was
originally proved by Mimachi-Yamada using differential equations and
it was done over finitely many variables.

Adopting the same idea as
in the Laplace-Beltrami like operator, we give an explicit action of Virasoro algebra on Jack functions
which confirms several conjectural formulas made by Sakamoto et al \cite{SSAFR}.
As a consequence
we obtain a simpler proof of
Mimachi-Yamada's result using the Feigin-Fuchs realization.

This paper is organized as follows. In section 2, we recall some
basic notions about symmetric functions. We introduce the
differential operator of Laplace-Beltrami type in section 3 and
compute its action on generalized homogeneous polynomials to give a
new characterization of Jack symmetric functions. In section 4 we derive
a raising operator formula for the action of Laplace-Beltrami operator.
Next in section 5 we derive several applications of our new differential operator: a
determinant formula for Jack symmetric functions, an iterative
formula for the transition matrix from generalized homogeneous
functions or monomial symmetric functions to Jack functions, and
explicit action of Virasoro operators on Jack functions. We also
show how our method can be used to give combinatorial formulas for
Jack symmetric functions and then for generalized
Littlewood-Richardson coefficients.

 \section{Jack functions}
   We first recall a few basic notions about symmetric functions \cite{M,S}.
   A partition $\lambda$ is a sequence of nonnegative
   integers, written usually in decreasing order as
   $\lambda=(\la_1,\lambda_2,\cdots,\lambda_s)$, sometimes also as $(1^{m_1}~2^{m_2}\cdots)$,
   where $m_i=m_i(\la)$ is the multiplicity of $i$ occurring in the parts of
    $\la$. The length of $\la$, denoted as $l(\lambda)$, is the number of non-zero
    parts in $\lambda$, and the weight $|\lambda|$ is $\sum_iim_i$. It is
    convenient to denote $m(\lambda)!=m_1!m_2!\cdots$, $z_\lambda=\prod_{i\geq 1}i^{m_i}m_i!$. A partition $\la$ of
weight $n$ is usually denoted by $\la\vdash n$. The set of all
partitions is
   denoted as $\mathcal {P}$. For two partitions
$\la=(\lambda_1,\lambda_2,\cdots)$ and $\mu=(\mu_1,\mu_2,\cdots)$ of
the same weight, we say that $\lambda\geq\mu$ if
$\lambda_1+\cdots+\lambda_i\geq\mu_1+\cdots+\mu_i$ for all
   $i$, this defines what's called dominance ordering. There is a canonical total ordering on
   $\mathcal{P}_n$, the reverse
lexicographic ordering. For $\lambda=(\lambda_1,\lambda_2,\cdots)$
and $\mu=(\mu_1,\mu_2,\cdots)$ in $\mathcal{P}_n$, we say that
$\lambda$ is greater than $\mu$ in reverse lexicographic ordering,
denoted as $\lambda>^L\mu$, if the first non-vanishing difference
$\lambda_i-\mu_i$ is positive, and $\lambda\geq^L\mu$ means
$\lambda>^L\mu$ or $\lambda=\mu$. Sometimes $\lambda$ is  identified
with its Young diagram $\lambda=\{(i,j)|1\leq i\leq l(\lambda),
1\leq j\leq\lambda_i\}$. Thus $\mu\subseteq\lambda$ if and only if
   $\lambda_i\geq\mu_i$ for all $i$, and in this case we define skew
   partition $\lambda/\mu$ as the set difference of $\lambda$ and
   $\mu$. We say that $\lambda/\mu$ is a horizontal-$n$ strip if it contains $n$ squares
   with no two squares in the same column.
    The conjugate of $\la$, $\lambda'=(\lambda'_1,\lambda'_2,\cdots)$,
   is a partition whose diagram is the transpose of the diagram of $\la$,
   hence $\la'_i$ is the number of the $j$'s such that $\la_j\geq i$.
    For square $s=(i,j)\in\la$, the lower hook-length $h_*^\lambda(s)$ is defined
to be
   $\alpha(\lambda_i-j)+(\lambda'_j-i+1)$, and the upper hook-length
$h^*_\lambda(s)=\alpha(\lambda_i-j+1)+(\lambda'_j-i)$.

   \begin{definition}\label{bottomnotation}
For a subset $S$ of partition $\la$, define $h_*^\la(S)$ as
$\prod_{s\in S}h_*^\la(s)$, and similarly for $h_\la^*(S)$. For
partitions $\mu\subseteq\la$, we say that a square $s$ is bottomed
if $s$ is in the column which contains at least one square of
$\la/\mu$, and it's un-bottomed otherwise. We denote $\mu_b$ (resp.
$\la_b$) as the set of the bottomed squares of $\mu$(resp. $\la$),
and $\mu_u$ (resp. $\la_\mu$) as the set of the un-bottomed squares
of $\mu$ (resp. $\la$).
\end{definition}

The ring $\Lambda$ of symmetric functions is a $\mathbb Z$-module
with basis $m_{\la}=\sum x^{\la_1}_{i_1}\cdots x_{i_k}^{\la_k}$,
$\la\in\mathcal {P}$. The power sum symmetric functions
$p_{\la}=p_{\la_1}\cdots p_{\la_k}$ form a basis of
$\Lambda_\mathbb{Q}=\Lambda\otimes_\mathbb{Z}\mathbb{Q}$.

Let $F=\mathbb{Q}(\alpha)$ be the field of rational functions in
indeterminate $\alpha$. The Jack polynomial \cite{Ja} is a special
orthogonal symmetric function of
$\Lambda_F=\Lambda\otimes_\mathbb{Z} F$ under the following inner
product. For two partitions $\la, \mu \in \mathcal P$ the (Jack)
scalar product on $\Lambda_F$ is given by
\begin{align} \label{def}
\langle p_{\la}, p_{\mu}\rangle=\delta_{\la,
\mu}\alpha^{l(\la)}z_\lambda,
\end{align}
where $\delta$ is the Kronecker symbol.

The Jack symmetric functions $P_{\lambda}(\alpha)$ for
$\la\in\mathcal{P}$ are defined by the following \cite{M}:
\begin{align*}
&P_{\la}(\alpha)=\sum_{\la\geq\mu}c_{\la
\mu}(\alpha)m_{\mu},\\
&\langle
P_{\la}(\alpha),P_{\mu}(\alpha\rangle=0~~\mbox{for}~\la\neq\mu,
\end{align*}
where $c_{\la \mu}(\alpha)\in F$ ($\la, \mu \in \mathcal P$) and
$c_{\la \la}(\alpha)=1$.

Defined by $\langle
Q_{\lambda}(\alpha),P_{\mu}(\alpha)\rangle=\delta_{\la,\mu}$, the
dual Jack function $Q_{\lambda}(\alpha)=\langle P_{\lambda},
P_{\lambda}\rangle^{-1}P_{\lambda}(\alpha)$. Another normalization
$J_\lambda(\alpha)$ of Jack symmetric function is also useful. Let
$$J_\lambda(\alpha)=\sum_{\nu\leq\lambda}v_{\lambda\nu}(\alpha)m_\nu,$$
with the normalization defined by
$v_{\lambda,(1^{|\lambda|})}=|\lambda|!$.

 The generalized homogeneous symmetric functions of $\Lambda_F$ are
defined by
\begin{equation}
q_{\la}(\alpha)=Q_{\la_1}(\alpha)Q_{\la_2}(\alpha)\cdots
Q_{\la_l}(\alpha),
\end{equation}
where $Q_{(n)}(\alpha)$, simplified as $Q_n(\alpha)$, is known as
\begin{equation}\label{E:homogeneous1}
Q_n(\alpha)=\sum_{\lambda\vdash
n}\alpha^{-l(\lambda)}z_\lambda^{-1}p_{\lambda}.
\end{equation}
Thus $Q_i(\alpha)=0$ for $i<0$ and $Q_0(\alpha)=1$. Its
generating function is:
$$Y(z)=exp\Big(\sum_{n=1}^\infty
\frac{z^n}{n\alpha}p_n\Big)=\sum_n Q_{n}(\alpha)z^n.$$

For convenience we may omit the parameter $\alpha$ in $Q_n(\alpha)$
and $q_\la(\alpha)$, simply write them as $Q_n$ and
$q_\la$.
 The
following theorem of Stanley will be needed in our paper.
\begin{theorem} \cite{S} \label{T:Stanley}
For partitions $\mu$ ,$\lambda$, and positive integer $n$, $\langle
J_nJ_\mu,J_\lambda\rangle\neq 0$ if and only if
$\mu\subseteq\lambda$ and $\lambda/\mu$ is a horizontal $n$-strip.
And in this case we have
 \begin{equation}
\langle
J_n(\alpha)J_\mu(\alpha),J_\lambda(\alpha)\rangle=h^\mu_*(\mu_u)h^*_\mu(\mu_b)\cdot
h^*_n(n)\cdot h^*_\lambda(\lambda_u)h^\lambda_*(\lambda_b).
\end{equation}
\end{theorem}

We also list some useful properties of $q_{\lambda}$ and $J_\lambda$
as follows.
\begin{lemma} \cite{S}\label{L:triangular}
For any partition $\lambda$, $\nu$, one has
\begin{align*}
&\langle q_{\la}(\alpha),m_{\nu}\rangle=\delta_{\la\nu},\\
&q_{\lambda}(\alpha)=Q_{\la}(\alpha)+\sum_{\mu>\lambda}c'_{\lambda\mu}Q_{\mu}(\alpha),\\
&J_\lambda(\al)=h_\la^*(\la)Q_\la(\al),\\
&J_\lambda(\al)=h^\la_*(\la)P_\la(\al),
\end{align*}
where 
$c'_{\lambda\mu}\in F$.
\end{lemma}

Knop and Sahi proved that (see also \cite {HHL}).
\begin{theorem}\cite{KS}\label{T:KSMac}
Let
$J_\lambda(\alpha)=\sum_{\mu\leq\lambda}v_{\lambda\nu}(\alpha)m_\mu,$
then we have $\frac{v_{\la\mu}(\alpha)}{m(\mu)!}\in
\mathbb{Z}_{\geq0}[\al].$
\end{theorem}

 In \cite{EJ}, some transition matrices were given using
combinatorial methods, the following is a special case of such
matrices.

\begin{proposition}\cite{EJ}\label{T:ptomu}
For $\mu\vdash n$, set
$m_\mu=\sum_{\lambda\geq\mu}T_{\mu\lambda}p_\lambda$, then we have
$T_{\mu\lambda}m(\mu)!\in\mathbb{Z}$ and
$T_{\mu,(n)}=(-1)^{l(\mu)-1}\frac{(l(\mu)-1)!}{m(\mu)!}$.
\end{proposition}

Combining Theorem \ref{T:KSMac} with Proposition \ref{T:ptomu}, it
is easy to see one of Stanley's conjectures:
\begin{theorem} \cite{S}
Set $J_\la(\al)=\sum_{\mu}c_{\la\mu}(\al)p_\mu$, then we have
$$c_{\la\mu}(\al)\in\mathbb{Z}[\al].$$
\end{theorem}

A direct consequence of this is
\begin{corollary}
$$C_{\mu\nu}^\la(\al)=\langle J_\mu(\alpha)
J_\nu(\alpha),J_\lambda(\alpha)\rangle\in\mathbb{Z}[\al].$$
\end{corollary}

\section{Laplace-Beltrami type operator for
Jack functions}
%

$\Lambda_F$ is a graded ring with gradation given by the degree, and
let $\Lambda_F(m)=\{f\in\Lambda_F|deg(f)=m\}$, thus
 $\Lambda_F=\oplus_{n=0}^{\infty}\Lambda_F(n)$. A linear operator
 $A$ is called a graded operator of degree $n$ if $A.\Lambda_F(m)\subset \Lambda_F(m+n)$.

 The Heisenberg algebra $H_\alpha$ is an infinite dimensional Lie algebra generated by
 $h_n$ ($n\neq 0$),
satisfying the relations:
\begin{align*}
[h_m, h_n]=m\alpha\delta_{m, -n}.
\end{align*}
For $n>0$, identifying $h_{-n}$ with $p_n$ in $\Lambda_F$, we have
the canonical representation of $H_\alpha$ on $\Lambda_F$ defined
by:
\begin{align*}
&h_n.v=n\alpha\frac{\partial}{\partial h_{-n}}(v),\\
&h_{-n}.v=h_{-n}v,
\end{align*}
for $n>0$ and $v\in\Lambda$.

The operator $h_n$ is then a graded
operator of degree $-n$.

On the ring $\Lambda_F$ of symmetric functions, we introduce the
following graded operator of degree $0$
$$
D(\alpha)=\sum_{i,j\geq 1}(i+j)\alpha
p_{i}p_{j}\frac{\partial}{\partial p_{i+j}}+\sum_{i,j\geq
1}ij\alpha^2 p_{i+j}\frac{\partial^2}{\partial p_{i}\partial
p_{j}}+\alpha(\alpha-1)\sum_{k\geq 1}k^2
p_{k}\frac{\partial}{\partial p_{k}},
$$
where the infinite sum is well-defined on each subspace
$\Lambda(m)$.

Using Heisenberg relations the operator $D(\alpha)$ can be rewritten
as
$$D(\alpha)=\sum_{i,j\geq 1}h_{-i}h_{-j}h_{i+j}+\sum_{i,j\geq
1}h_{-(i+j)}h_{i}h_{j}+(\alpha-1)\sum_{i\geq 1}ih_{-i}h_i.
$$
\begin{remark} A similar operator
was used in physics literature (eg. \cite{AMOS}, \cite{I}) to study
Virasoro constraints. When $\alpha=1$, the operator $D(1)$ was also
studied by \cite{FW} for the Schur case. We will see that the third
term is crucial in the Jack case.
\end{remark}

We now give a characterization of Jack functions using the Laplace-Beltrami operator.

\begin{theorem} \label{Dalphatheorem}
For $\la=(\la_1,\la_2,\cdots)\in \mathcal P$,
$Q'_{\lambda}(\alpha)=Q_{\lambda}(\alpha)$ if and only if the
following properties are satisfied:
\begin{align*}
&\mbox{(1)}~~\mbox{There are rational functions~}C_{\la\mu}(\al)
\mbox{~of~}\alpha\mbox{~with~}C_{\la\la}(\al)=1\mbox{~such that}\\
&
 \qquad\qquad\qquad\qquad Q'_\la(\al)=\sum_{\mu\geq\la}C_{\la\mu}(\al)q_\mu(\al),\\
&\mbox{(2)}~~Q'_\lambda(\alpha) \mbox{~is an eigenvector
for~}D(\alpha) .
\end{align*}
Moreover,
$D(\alpha).Q_\lambda(\alpha)=e_\lambda(\alpha)Q_\lambda(\alpha)$,
where
$$e_\lambda(\alpha)=\alpha^2\sum_i\lambda_i^2+\alpha(|\la|-2\sum_ii\lambda_i).$$
\end{theorem}


\begin{definition}
For an operator $S$ on $\Lambda_F$, define the conjugate of $S$,
denoted as $S^*$, by $\langle S.u,v\rangle=\langle u,S^*.v\rangle$
for all $u,v\in\Lambda_F$. $S$ is called self-adjoint, if $S=S^*$.
Let $S$ be a degree $0$ graded operator on $\Lambda_F$, we say that
$S$ is a raising operator on $q_{\la}$'s if
$S.q_\lambda=\sum_{\mu\geq\lambda}C_{\lambda\mu}q_\mu$ for every
$\lambda\in\mathcal {P}$.
\end{definition}

It's easy to prove the following Lemma, mainly using Lemma
\ref{L:triangular}.
\begin{lemma} \label{L:OHJAE}
A graded operator $S$ on $\Lambda_F$ has Jack functions as its
eigenvectors if and only if $S$ is a raising operator on $q_\la$'s
and $S$ is self-adjoint.
\end{lemma}


To prove Theorem \ref{Dalphatheorem}, we will show that
$D(\alpha)$ is a raising operator on $q_\la$'s as $D(\alpha)$ is
obviously self-adjoint. First, let us look at the one row case.

\begin{lemma} The Jack functions $Q_n(\alpha)$'s are eigenvector of
$D(\alpha)$, explicitly we have
\begin{equation}
D(\alpha).Q_n(\alpha)=(\alpha^2n^2-n\alpha)Q_n(\alpha).
\end{equation}
\end{lemma}

\begin{proof}
Recall that $Q_n(\alpha)=\sum_{\lambda\vdash
n}\alpha^{-l(\lambda)}z_\lambda^{-1}p_{\lambda}$. For
$\mu=(1^{m_1}2^{m_2}\cdots)$, the coefficient of $p_\mu$ in
$D(\alpha).Q_n(\alpha)$ is
\begin{align} \nonumber
&(\alpha-1)\sum_{k\geq 1}\alpha k^2m_k \alpha^{-l(\mu)}z_\mu^{-1}\\
\nonumber &+ \sum_{i,j\geq 1}\alpha
(i+j)(m_{i+j}+1)\alpha^{-(l(\mu)-1)}z_\mu^{-1}im_ij(m_j-\delta_{i,j})(i+j)^{-1}(m_{i+j}+1)^{-1}\\
\nonumber &+ \sum_{i,j\geq 1}\alpha i\alpha
j(m_i+1)\alpha^{-(l(\mu)+1)}z_\mu^{-1}i^{-1}j^{-1}
(m_j+1+\delta_{i,j})^{-1}(i+j)m_{i+j}\\ \nonumber
&=\alpha^{-l(\mu)}z_\mu^{-1}\Big[(\alpha-1)\alpha\sum_{k\geq
1}k^2m_k+\sum_{i,j\geq1}\alpha^2im_ij(m_j-\delta_{i,j})+\sum_{i,j\geq 1}\alpha(i+j)m_{i+j}\Big]\\
\nonumber
&=\alpha^{-l(\mu)}z_\mu^{-1}\Big(\alpha^2|\mu|^2-\alpha |\mu|\Big).
\end{align}
\end{proof}

Note that the first and third summand of $D(\alpha)$ are derivations
on $\Lambda_F$. The second term is a second order differential
operator. The following lemma, Lemma \ref{act-on-prod}, assists in
computing its action on products. First we need the following
corollary of Proposition \ref{T:ptomu}.

\begin{lemma}\label{L:pnasqlambda}
Let $p_n=\sum_{\mu\vdash n} a_{n,\mu}q_\mu(\alpha)$, then we have
$$a_{n,\mu}=n\alpha(-1)^{l(\mu)-1}(l(u)-1)!/m(\mu)!.$$
\end{lemma}
\begin{proof}
\begin{align*}
a_{n,\mu}&=\langle h_{-n}.1,m_\mu\rangle=n\alpha\langle
1,h_{n}.m_\mu\rangle=\langle1,\frac{\partial}{\partial
h_{-n}}m_\mu\rangle=n\alpha
 T_{\mu,(n)}\\
 &=n\alpha(-1)^{l(\mu)-1}(l(u)-1)!/m(\mu)!.
\end{align*}
\end{proof}
Applying $h_m$ to two sides of $Y(z)=exp\Big(\sum_{n=1}^\infty
\frac{z^n}{n\alpha}h_{-n}\Big)=\sum_n Q_{n}(\alpha)z^n$, we have the
following lemma about the action of $h_n$ on one-row Jack functions.
\begin{lemma}\label{L:h_n.Q_m}
For positive integer $m$ and integer $n$, we have $h_m.Q_n=Q_{n-m}$.
\end{lemma}

\begin{lemma} \label{act-on-prod}
For positive integers $m$ and $n$ with $m\geq n$, set
$$A_{m,n}=\sum_{i,j\geq
1}h_{-(i+j)}(h_i.Q_{m})(h_j.Q_{n})=\sum_{\la}b_\la q_\la,$$ then
$b_\la=[m'(1-\delta_{m,m'})-n']\alpha$ for $\la=(m',n')\geq (m,n)$,
and $b_\la=0$ otherwise.
\end{lemma}

\begin{proof}
By Lemma \ref{L:h_n.Q_m}, we have
$$A_{m,n}=\sum_{i,j\geq 1}h_{-(i+j)}q_{m-i}q_{n-j}.$$ Let's consider
the three cases of $\la$: $l(\la)=1$, $l(\la)=2$, and $l(\la)\geq
3$.

First, for $l(\la)=1$, we have $\la=(m+n,0)\geq (m,n)$. In the
summation of $A_{m,n}$, only $(i,j)=(m,n)$ contributes to $q_\la$.
By Lemma \ref{L:pnasqlambda}, the coefficient is
$a_{(m+n),(m+n)}=(m+n)\alpha$. This is in accord with the statement
of the lemma.

Second, for $l(\la)=2$, set $\lambda=(m',n')$. If we don't have
$\la\geq (m,n)$, then we have $m>m'>n'>n$ or $m>m'=n'>n$. If
$m>m'>n'>n$, the coefficient of $q_\lambda$ in $A_{m,n}$ is
\begin{align*}
&a_{n',n'}+a_{m',m'}+a_{m'+n',(m',n')}=n'\alpha+m'\alpha+(m'+n')\alpha(-1)=0.
\end{align*}
It can be found similarly that the coefficient is also zero if
$m>m'=n'>n$. If we do have $\la\geq (m,n)$, it can be found
similarly that $(m'-n')\alpha$ if $m'>m\geq n>n'$ , and the
coefficient is $-n'\alpha$ if $m'=m\geq n=n'$.

Third, for $l(\la)\geq3$, set
$\lambda=(\lambda_1,\lambda_2,\cdots)=(1^{m_1}2^{m_2}\cdots)$. The
coefficient of $q_\lambda$ is
\begin{align}
&\alpha(-1)^{l(\lambda)}(l(\lambda)-3)!(m(\lambda)!)^{-1}\Big[-(m+n)(l(\lambda)-1)(l(\lambda)-2)\\
\nonumber
&+\sum_{j=1}^{n-1}(m+n-j)(l(\lambda)-2)m_j+\sum_{i=1}^{m-1}(m+n-i)(l(\lambda)-2)m_i\\
\nonumber
&+\sum_{i=1}^{m-1}\sum_{j=1}^{n-1}(m+n-i-j)(-1)m_i(m_j-\delta_{i,j})\Big].
\end{align}
For convenience we denote the term inside of the square bracket as
$A$. The four summands in $A$ correspond to the four kinds of
assignment of $(i,j)$, with the first one corresponding to
$(i,j)=(m,n)$, the second corresponding to $i=m$,$j=1,\cdots,n-1$, the
third to $i=1,\cdots,m-1$, $j=n$, and the last to $i=1,\cdots,m-1$,
$j=1,\cdots,n-1$. We need to prove that $A=0$.

In the first subcase that $\lambda_1\leq n-1$,  $A$ is equal to
\begin{align*}
A_1&=-(m+n)(l(\la)-1)(l(\la)-2)+2(m+n)(l(\la)-1)(l(\la)-2)\\
&\qquad\qquad\qquad-(m+n)(l(\la)-1)(l(\la)-2)\\
&=0,
\end{align*}
where we use the property that $\sum_im_i=l(\la)$, and
$\sum_iim_i=m+n$.

In the second subcase that $\la_1\geq m$, 
%
 and the third subcase that $m>\la_1\geq n$, 
$A$ can be proved similarly to be zero.
\end{proof}

\section{Raising operator formula for Laplace-Beltrami operator}

The differential operator $D(\alpha)$ acts triangularly on the generalized homogeneous polynomials.

\begin{proposition}\label{Dalphaproposition}
The action of $D(\alpha)$ on $q_\la$ is given explicitly as follows.
\begin{align}
&D(\alpha).q_\lambda(\alpha)\\
\nonumber &=e_\lambda(\alpha)q_\lambda(\alpha)+2\alpha
\sum_{i<j}\sum_{k\geq 1}(\lambda_i-\lambda_j+2k)q_{\lambda_1}\cdots
q_{\lambda_i+k}\cdots q_{\lambda_j-k}\cdots.
\end{align}
\end{proposition}

\begin{proof}
%

Write $D(\alpha)=A(\alpha)+B(\alpha)$ with $B(\alpha)=\sum_{i,j\geq
1}h_{i}h_{j}h_{-(i+j)}$. Then $A(\alpha)$ is a derivation on $V$.
For $q_\lambda=q_{\lambda_1}q_{\lambda_2}\cdots q_{\lambda_s}$, we
have:
\begin{align*}
B(\alpha).q_\lambda&
\nonumber  =\sum_lq_{\lambda_1}\cdots
(B(\alpha).q_{\lambda_l})\cdots q_{\lambda_s}\\
&\qquad\qquad\qquad+\sum_{l\neq m}\sum_{i,j\geq
1}h_{-(i+j)}q_{\lambda_1}\cdots
(h_{i}.q_{\lambda_l})\cdots(h_{j}.q_{\lambda_m})\cdots
q_{\lambda_s}.
\end{align*}

Combining the action of $A(\alpha)$ and $B(\alpha)$  we have
\begin{align*}
D(\alpha).q_\lambda&
=\sum_lq_{\lambda_1}\cdots (D(\alpha).q_{\lambda_l})\cdots
q_{\lambda_s}\\ &\quad\qquad\quad+2\sum_{l<m}\sum_{i,j\geq
1}h_{-(i+j)}q_{\lambda_1}\cdots q_{\lambda_l-i}\cdots
q_{\lambda_m-j}\cdots
q_{\lambda_s}\\
\end{align*}
Applying Lemma (\ref{act-on-prod}) to it finishes the proof.
\end{proof}

\begin{proposition}\label{Q'}
There exists a unique family of rational functions
$\{C_{\la\mu}(\al)|\mu\geq\la\}$  such that

(1) $C_{\la\la}(\al)=1$ for all $\la\in\mathcal{P}$,

(2) $\sum_{\mu\geq\la}C_{\la\mu}(\al)q_\mu(\al)$, denoted as
$Q'_\la(\al)$, is an eigenvector for $D(\alpha)$ for each
$\la\in\mathcal{P}$.
\end{proposition}

\begin{proof}For each partition $\la$, we use induction on the
dominance ordering $\geq$ to prove that there is a unique family
$\{C_{\la\mu}(\al)|\mu\geq\la\}$, such that the two properties are
satisfied. First note that by setting $C_{\la\la}(\al)=1$ we see the
coefficient of $q_\la(\al)$ in the following expression is zero:
\begin{equation}\label{E:Q'alpha}
e_\la(\al)\sum_{\xi\geq\la}C_{\la\xi}(\al)q_\xi(\al)-\sum_{\xi\geq\la}C_{\la\xi}(\al)D(\al).q_\xi(\al).
\end{equation}
This is due to the fact that the coefficient of $q_\la$ in
$D(\al).q_\la$ is $e_\la(\al)$. Now assume that for each $\nu$ such
that $\la\leq\nu<\mu$, $C_{\la\nu}(\al)$ is already found such that
the coefficients of $q_\nu(\al)$ in equation (\ref{E:Q'alpha}) are
zero. We will see that $C_{\la\mu}(\alpha)$ is uniquely determined
such that the coefficient of $q_\mu(\al)$ in equation
(\ref{E:Q'alpha}) is also zero. This means that we must have the
following equation
\begin{align*}
e_\la(\al)C_{\la\mu}(\al)-\sum_{\la\leq\nu\leq\mu}C_{\la\nu}(\al)R_{\nu\mu}(\al)=0,
\end{align*}
where $R_{\nu\omega}(\alpha)$ is defined by
$D(\al).q_\nu(\al)=\sum_{\omega\geq\nu}
R_{\nu\omega}(\al)q_\omega$ (thus we have
$R_{\mu\mu}(\al)=e_\mu(\al)$). By the inductive hypothesis,
$C_{\la\nu}(\al)$ have already been found except $\nu=\mu$. Then we
 solve the equation to determine $C_{\la\mu}(\al)$:
\begin{align*}
C_{\la\mu}(\al)=\frac{\sum_{\nu<\mu,
|\nu|=|\la|}C_{\la\nu}(\al)R_{\nu\mu}(\al)}{e_\la(\al)-e_\mu(\al)}.
\end{align*}
Here we note that $e_\la(\al)-e_\mu(\al)\neq0$ as $\la<\mu$ by the
following well-known Lemma \ref{L:eigenvalue}. This finishes the
existence and uniqueness of the family of
$C_{\la\mu}(\al)$'s.
\end{proof}

\begin{lemma}\label{L:eigenvalue}
For two partitions $\la$, $\mu$ with $\la>\mu$, we have
$e_\la(\al)\neq e_{\mu}(\al)$.
\end{lemma}

\medskip

Now it is clear that Theorem \ref{Dalphatheorem} follows by combining
Lemma \ref{L:OHJAE}, Proposition
\ref{Dalphaproposition} and Proposition \ref{Q'}.
\medskip

We can rewrite the action of the differential operator $D(\alpha)$
on the basis of $q_{\la}$. Recall that a raising operator is a
product of simple raising operator $R_{ij}$ on partitions:
\begin{equation}
R_{ij}\la=(\la_1, \cdots, \la_i+1, \cdots, \la_j-1, \cdots, \la_l)
\end{equation}
We also define the action of a raising operator on $q_{\la}$ by its
action on the associated partition $\la$, i.e.,
$R_{ij}q_{\la}=q_{R_{ij}\la}$.

Using the raising operator on the basis $\{q_{\la}\}$, we can
rewrite the action easily as follows.

\begin{corollary} \label{C:action-of-D}
The action of $D(\alpha)$ on the basis $q_{\lambda}$ is given by
\begin{align*}
D(\alpha)q_{\lambda}&=G_{\la}(R_{ij})q_{\la}\\
&=\left[e_{\lambda}(\alpha)+2\alpha\sum_{1\leq i<j\leq
l(\la)}\left(\frac{(\la_i-\la_j)R_{ij}}{1-R_{ij}}+\frac{2R_{ij}}{(1-R_{ij})^2}
\right)\right]q_{\la}.
\end{align*}
\end{corollary}
\begin{proof} Note that
$$R_{ij}^kq_{\la}=q_{\la_1}\cdots q_{\la_i+k}\cdots q_{\la_j-k}\cdots q_{\la_l}, $$
and $R_{ij}^kq_{\la}=0$ whenever $k>min(\la_i, \la_j)$. It follows
that
\begin{align*}
&D(\alpha)q_\lambda\\
&=e_\lambda(\alpha)q_\lambda+2\alpha
\sum_{i<j}\sum_{k\geq 1}(\lambda_i-\lambda_j+2k)R_{ij}^kq_{\lambda}\\
&=e_\lambda(\alpha)q_\lambda+2\alpha
\sum_{i<j}\left((\lambda_i-\lambda_j)\frac{R_{ij}}{1-R_{ij}}+\frac{2R_{ij}}{(1-R_{ij})^2}\right)q_{\lambda},
\end{align*}
where we used $\sum_{k\geq 1}kR^k=\frac R{(1-R)^2}$.
\end{proof}

When $\la$ is a rectangle, we have
\begin{equation}
D(\alpha)q_{\lambda}=\left[e_{\lambda}(\alpha)+\binom{\l(\al)}2\frac{4R_{ij}}{(1-R_{ij})^2}
\right]q_{\la},
\end{equation}
where $i<j$ is a fixed pair. The formula clearly shows that the case
of rectangular shapes are special.

\begin{remark}We will derive a raising-operator-like formula (see corollary
\ref{C:rising-operator-like formula}) for the Jack symmetric
functions at the end of next section.
\end{remark}

\section{Applications of the differential operator $D(\alpha)$}

 Note that $D(\alpha)$ is
self-adjoint and the bases $q_\la(\alpha)$'s and $m_\la$'s are dual,
most properties involving
 with $D(\alpha)$ and $q_\la(\alpha)$'s can be passed to those about
 $D(\alpha)$ and $m_\la$'s. We list some  of these properties with
proofs omitted.


\medskip





\begin{proposition}\label{P'}
There exists a unique family of rational functions
$\{B_{\la\mu}(\al)|\mu\geq\la\}$ of $\al$ such that

(1) $B_{\la\la}(\al)=1$ for all $\la\in\mathcal{P}$,

(2) $\sum_{\mu\leq\la}B_{\la\mu}(\al)m_\mu$, denoted as
$P'_\la(\al)$, is an eigenvector for $D(\alpha)$  for all
$\la\in\mathcal{P}$.
\end{proposition}
\begin{theorem} \label{Dalphatheorem'}
For $\la=(\la_1,\la_2,\cdots)\in \mathcal P$,
$P'_{\lambda}(\alpha)=P_{\lambda}(\alpha)$ if and only if the
following properties are satisfied:

(1) There are rational functions $D_{\la\mu}(\al)$ of $\alpha$ with
$D_{\la\la}(\al)=1$ such that

$$ P'_\la(\al)=\sum_{\mu\leq\la}D_{\la\mu}(\al)m_\mu,$$

(2) $P'_\lambda(\alpha)$ is an eigenvector for $D(\alpha)$.

\end{theorem}
\medskip

We know that the symmetric functions constructed in Proposition
\ref{Q'} and \ref{P'} are $Q_\lambda(\alpha)$ and
$P_\lambda(\alpha)$ respectively, thus they are dual to each other.
We can also prove this directly from their constructions.
\begin{proposition}
We have $\langle
P'_\la(\alpha),Q'_\mu(\alpha)\rangle=\delta_{\la,\mu}$.
\end{proposition}

\begin{proof}
If $\la=\mu$, $\langle
P'_\la(\alpha),Q'_\mu(\alpha)\rangle=1$ for that $\langle
m_\la,q_\mu(\alpha)\rangle=\delta_{\la,\mu}$. If $\la\neq\mu$,
consider two cases. In the first case that $\la$ and $\mu$ are
comparable, we have $e_\la(\al)\neq e_\mu(\al)$, and thus
\begin{align*}
&e_\la(\al)\langle P'_\la(\alpha),Q'_\mu(\alpha)\rangle=\langle e_\la(\al)P'_\la(\alpha),Q'_\mu(\alpha)\rangle\\
&=\langle D(\al).P'_\la(\alpha),Q'_\mu(\alpha)\rangle=\langle P'_\la(\alpha),D(\al).Q'_\mu(\alpha)\rangle\\
&=e_\mu(\al)\langle P'_\la(\alpha),Q'_\mu(\alpha)\rangle.
\end{align*}
Hence $\langle P'_\la(\alpha),Q'_\mu(\alpha)\rangle=0$ is immediate.

In the second case that $\la$ and $\mu$ are incomparable, we do not
have a $\nu$ such that $\nu\leq\la$ and $\nu\geq\mu$. Because that
will lead to $\la\geq\mu$ which is a contradiction. Thus every term
in $P'_\la(\alpha)$ is orthogonal to the terms in $Q'_\mu(\alpha)$,
and $\langle P'_\la(\alpha),Q'_\mu(\alpha)\rangle=0$.
\end{proof}

\subsection{On the action of the Virasoro algebra}
The Virasoro algebra is the infinite dimensional Lie algebra
generated by $L_n$ and the central element $c$ subject to the relations:
\begin{align} \label{F:Vir}
[L_m, L_n]&=(m-n)L_{m+n}+\frac{m^3-1}{12}\delta_{m, -n}c, \\
[L_n, c]&=0.
\end{align}
The Feigin-Fuchs realization of Virasoro algebra on $\Lambda$ can be
formulated as follows.
\begin{align*}
&L_n=\frac{1}{2}\sum_{m\in\mathbb{Z}}:a_{n-m}a_m:-\al_0(n+1)a_n,\\
&c=1-12\al_0^2=13-6(\al+\frac{1}{\al}),\\
\end{align*}
where the Heisenberg generators $a_n$ are given by
\begin{align*}
&a_{-m}=(-1)^{m-1}\frac{1}{\sqrt{2\al}}h_{-m}, \qquad m>0,\\
&a_m=(-1)^{m-1}\sqrt{\frac{2}{\al}}h_m, \qquad m>0,\\
&a_0=\al'Id,\\
& \al_0=(\frac{\sqrt{2\al}}{2}-\frac{1}{\sqrt{2\al}}),
\end{align*}
where we assume $\alpha>0$.
The new operators satisfy the standard relations:
\begin{equation}
[a_m, a_n]=m\delta_{m, -n}I.
\end{equation}

For $n>0$, we have $L_{n+2}=(-1)^n(n!)^{-1}(adL_1)^nL_2$ by
(\ref{F:Vir}), thus to know the action of $L_n$ we only need to know
that of $L_1$ and $L_2$, which is explicitly given as follows:
\begin{align}\label{F:L1}
&L_1=-\al^{-1}M_1+\big(\al'\sqrt{\frac 2{\al}}-2(1-\frac{1}{\al})\big)h_1,\\
\label{F:L2}
&L_2=\al^{-1}M_2+(3\al_0-\al')\sqrt{\frac 2{\al}}h_2+\al^{-1}h_1h_1,
\end{align}
where for non-negative integer $n$, we define
\begin{align*}
M_{-n}&=\sum_{i\geq1}h_{-i-n}h_i,\\
M_n&=\sum_{i\geq1}h_{-i}h_{i+n},
\end{align*}
thus
$M_{-n}^*=M_n$ for $n\in\mathbb{Z}$.

Note that we have $h_{2}=\al^{-1}(J_2^*-h_{1}^2)$ and $h_{1}=J_1^*$,
where $J_n^*$ is the conjugate of $J_n$, which can be taken as a
multiplication operator.  The action of $J_n$ on Jack functions,
known as Pieri formula, was discovered by Stanley. Also note that
$L_0$ is a scalar multiplier on each $\Lambda_F(m)$,
$L_{-(n+2)}=n!)^{-1}(adL_{-1})^n.L_{-2}$, and $h_{-n}^*=h_n$, to
know the action of Virasoro algebra on Jack functions, we only need
to know those of $M_1$ and $M_2$. This is given by the following
proposition.

\begin{proposition}\label{P:M1M2}
For any pair of partitions $\mu,\la$, we have
\begin{align}\label{F:M1}
\langle M_1.J_\la,J_\mu\rangle
&=(2\al)^{-1}(e_\la(\al)-e_\mu(\al)-e_{(1)}(\al))\langle
J_1^*.J_\la,J_\mu\rangle,\\ \label{F:M2} \langle
M_2.J_\la,J_\mu\rangle
&=(4\al^2)^{-1}(e_\la(\al)-e_\mu(\al)-e_{(2)}(\al))\langle
J_2^*.J_\la,J_\mu\rangle\\
&\qquad\qquad\hskip 1.5in -\al^{-1}\langle
M_1J_1^*.J_\la,J_\mu\rangle. \nonumber
\end{align}
\end{proposition}

\begin{proof}
As in the proof of Proposition \ref{Dalphaproposition},
we have
\begin{align*}
D(\alpha).(J_\mu
J_\nu)&=(D(\alpha).J_\mu)J_\nu+J_\mu(D(\alpha).J_\nu)\\
&\hskip 1in +2\sum_{j\geq1}(M_{-j}.J_\mu)(h_j.J_\nu).
\end{align*}
Notice that $h_n.J_m(\alpha)=\frac{m!}{(m-n)!}\al^nJ_{m-n}(\al)$,
and that $D(\al).J_\la=e_\la(\al)J_\la$ for any partition $\la$.
Combining these into the case of $n=1,2$ in the following equation
finishes the proof:
$$\langle D(\alpha).(J_\mu J_n),J_\la\rangle=\langle J_\mu
J_n,D(\alpha).J_\la\rangle.$$
\end{proof}

\begin{remark}
 In \cite{SSAFR}, the action of $L_n$ on Jack function was conjectured in
an iterative formular. Our action of $M_1$ and $M_2$ given
in Proposition \ref{P:M1M2} \label{P:M1M2} partly confirms 
their formula.
\end{remark}

It is known \cite{FF} that the subspace spanned by the singular vectors of fixed degree
is of dimension one, thus the following recovers the result of
\cite{MY}.
\begin{corollary}
For a partition $\la$, $J_\la$ is a singular vector of the
representation of the Virasoro algebra if and only if $\la=(r^s)$
and $\al'=(r+1)\sqrt{2\al}/2-(1+s)/\sqrt{2\al}$ for some pair of
positive integers $(r,s)$.
\end{corollary}

\begin{proof}
To prove the necessity, we have
$$\langle L_1.J_\la,J_\mu\rangle=\phi(\la,\mu,\al')\langle
J_1^*.J_\la,J_\mu\rangle,$$ where
$\phi(\la,\mu,\al')=-(2\al^2)^{-1}(e_\la(\al)-e_\mu(\al)-e_{(1)}(\al))+(\al'\sqrt{2/\al}-2(1-1/\al))$.
Note that the Jack functions in the expansion of $J_1^*.J_\la$ are
labeled by partitions coming from $\la$ by removing one of the
square. If $\la$ is not of rectangular shape, there would be at
least two such partitions, say $\mu^1$ and $\mu^2$ with
$\mu^2<\mu^1$. Thus $\phi(\la,\mu^1,\al')\neq\phi(\la,\mu^2,\al')$,
and at least one of $\langle L_1.J_\la,J_\mu^1\rangle$ or $\langle
L_1.J_\la,J_\mu^2\rangle$ is non-zero, which means $J_\la$ is not a
singular vector. This proves that if $J_\la$ is a singular vector,
$\la=(r^s)$ for a pair of positive integers $(r,s)$. We then compute
the action of $L_1$ on $J_{(r^s)}$ using \ref{F:L1}, \ref{F:M1}, and
Pieri formula \cite{S} for the action of $J_1^*$. After some
computation, one can see that the unique value of $\al'$ that makes
$L_1.J_{(r^s)}$ vanished is $(r+1)\sqrt{2\al}/2-(1+s)/\sqrt{2\al}$ .

To prove the sufficiency, we only need to verify that both the
actions of $L_1$ and $L_2$ on $J_{(r^s)}$ lead to zero. The
computation about $L_1$ is already done above, while the action of
$L_2$ invokes \ref{F:L2}, \ref{F:M1}, \ref{F:M2}, and Pieri formula
for the action of $J_1^*$ and $J_2^*$. We can show that
$L_2.J_{(r^s)}$ also vanishes.
\end{proof}

\subsection{Determinant formulae for Jack symmetric functions}
We use the proof of Theorem \ref{Dalphatheorem} to find new
determinant formulae for Jack symmetric functions. First, note that
the operator $\sum_{k\geq 1}h_{-k}h_k$ is a scalar multiplier on
each $\Lambda(m)$, for simplicity we can replace the operator
$D(\alpha)$ with $D'(\alpha)=(2\alpha)^{-1}D(\alpha)-\sum_{k\geq
1}h_{-k}h_k$, we then have
$$D'(\alpha).Q_\lambda(\alpha)=e'_\lambda(\alpha)Q_\lambda(\alpha),$$
where
\begin{equation}\label{eq:eprime}
e'_\lambda(\alpha)=\frac{1}{2}\alpha\sum_i\lambda_i^2-\sum_ii\lambda_i
\end{equation}
for $\lambda=(\lambda_1,\lambda_2,\cdots)$. Note that $\la>\mu$
implies that $e'_\la(\al)-e'_\mu(\al)$ is of the form $a\al+b$, with
$a,b\in\mathbb{Z}_{>0}$. Next, we notice that the action of
$D(\alpha)$ or $D'(\alpha)$ on $q_\la(\alpha)$ can be refined. For
this purpose we have the following.

\begin{definition}\label{D:risingfiltration}
Let $\lambda=(\lambda_1,\lambda_2,\cdots,\lambda_s)$ be a partition,
assume that $i<j$ and $\lambda_j\geq k>0$, we define the action of
$r_j^i(k)$ on $\lambda$ as 
moving $k$ squares from the $j$th row to the $i$th row, then
rearranging the rows in decreasing order to get a new partition, i.e.
$r_j^i(k).\lambda$ is the rearrangement of
$(\lambda_1,\cdots,\lambda_i+k,\cdots,\lambda_j-k,\cdots,\lambda_s)$
in decreasing order. We call $r_j^i(k)$ a {\it moving up operator} for
$\lambda$, and define the moving up of $\lambda$ as the set
$$M^*(\lambda)=\{r_j^i(k).\lambda| r_j^i(k)~\mbox{is a moving up operator
of}~ \lambda\}.$$ We also define the moving down of $\mu$ as the set
$M_*(\mu)=\{\lambda| \mu\in M^*(\lambda)\}.$
\end{definition}

\medskip

For later use, we would like to give a refinement of Proposition
\ref{Dalphaproposition} using the deformation $D'(\alpha)$ of
$D(\alpha)$.

\begin{lemma}\label{L:actionofD'alpha}
Let $\lambda=(\lambda_1,\lambda_2,\cdots)$, we have
\begin{align*}
&D'(\alpha).q_\lambda(\alpha)=\sum_{\mu}r_{\lambda\mu}q_\mu(\alpha),\\
&D'(\alpha).m_\mu=\sum_{\la}r_{\lambda\mu}m_\lambda,
\end{align*}
where for $\mu=r_j^i(k).\lambda$,
$$r_{\lambda\mu}=(1+\delta_{\lambda_i\lambda_j})^{-1}m_{\lambda_i}(\la)
(m_{\lambda_j}(\la)-\delta_{\lambda_i\lambda_j})(\lambda_i-\lambda_j+2k),$$
 $r_{\nu\nu}=e'_\nu(\alpha)$ and $r_{\lambda\mu}=0$ otherwise.
\end{lemma}

\begin{proof}
The first equality is coming from Proposition
\ref{Dalphaproposition}. Noting that $D'(\alpha)$ is the
self-adjoint and that $m_\lambda$'s is dual to $q_\lambda$'s, we
have the second equality.
\end{proof}

In \cite{LLM}, Jack function was given in a determinant of a matrix
with entries being monomials, and equivalently a recursion formula
is given. In the following we have a similar formula expressing Jack
function as a determinant in terms of $q_\la$'s. We remark that we
can easily find the formula in terms of monomial symmetric functions
as well in a different way.

 For a partition $\la$, let $\la=\mu^1<^L\mu^2<^L
\cdots<^L\mu^s$ be all the partitions greater than $\la$,
arranged in lexicographic order. Set matrix $M_\la=(r_{ij})_{s\times
s}$, where $r_{ij}=r_{\mu^j\mu^i}$ as defined in Lemma
\ref{L:actionofD'alpha}, we have

\begin{theorem}
Set $Q_\la(\al)=\sum_iC_{\la\mu^i}(\al)q_{\mu^i}(\al)$, then the
vector
$$X_\la=(C_{\la\mu^1}(\al),C_{\la\mu^2}(\al),\cdots,C_{\la\mu^s}(\al))^t$$
satisfies
$$M_\la X_\la=e'_\la(\al)X_\la.$$ And we have the
determinant formula for Jack functions:
$$Q_\la(\al)=c_\la detN_\la,$$
where $c_\la=\prod_{i=2}^s(e'_{\mu^i}(\al)-e'_{\mu^1}(\al))^{-1}$,
and $N_\la$ is the matrix $M_\la-e'_\la(\al)Id$ with the first row
replaced by $(q_{\mu^1}(\al),q_{\mu^2}(\al),\cdots,q_{\mu^s}(\al))$.
\end{theorem}

\begin{proof}
Evaluating the coefficient of $q_{\mu^k}(\al)$ in
$$e'_\la(\al)Q_\la(\al)=\sum_iC_{\la\mu^i}(\al)D'(\al).q_{\mu^i}(\al),
$$
by Lemma \ref{L:actionofD'alpha}, we have
$$\sum_iC_{\la\mu^i}r_{\mu^i\mu^k}=e'_\la(\al)C_{\la\mu^k}.$$
This is the $k$th row of $M_\la X_\la=e'_\la(\al)X_\la$. Thus
$X_\la$ is a solution to the system of linear equations $(M_\la
-e'_\la(\al)Id)X=0$, note that the coefficient matrix $A=(M_\la
-e'_\la(\al)Id)$ has co-rank 1 and its first row is zero, an
elementary result of linear algebra says that the solutions are
proportional to $(A_{11},A_{12},\cdots,A_{1s})^t$, where $A_{ij}$ is
the algebraic cofactor of $a_{ij}$ in $A=(a_{ij})_{s\times s}$. Note
also that $A_{1k}=(N_\la)_{1k}$, and consider the coefficient of
$q_\la$ proves the second equation.
\end{proof}

\subsection{Generalized rising operator formula for Jack functions}
An equivalent form of the determinant formula is the following
iterative formula for the coefficients of Jack functions in terms of
generalized homogeneous functions. This formula is sometimes more
convenient to use.

\begin{theorem}\label{T:iteration}
Let $Q_\lambda(\alpha)=\sum_{\mu\geq\la}C_{\la\mu}(\al)q_\mu(\al)$,
for $\mu>\la$ we have
\begin{align*}
C_{\la\mu}(\al)=\frac{\sum_{\la\leq\nu\in
M_*(\mu)}C_{\la\nu}(\al)r_{\nu\mu}}{e'_\la(\al)-e'_\mu(\al)}.
\end{align*}
\end{theorem}

 As an
application of Theorem \ref{T:iteration}, we have the following
explicit formula for two-row or two-column Jack functions.
\begin{proposition}\cite{JJ,S}\label{P:2rowcolumn}
For $\lambda^0=(r,s)$, with $r\geq s$, $a=r-s$, set
$\lambda^i=(r+i,s-i)$ for $0\leq i\leq s$. We have
\begin{align}
&Q_{\lambda^0}(\alpha)=\sum_{s\geq i\geq0}a_i(\alpha)q_{\lambda^i}(\alpha),\\
&J_{(\lambda^0)'}(\alpha)=\sum_{s\geq
i\geq0}b_{s-i}(\alpha)m_{(\lambda^i)'},
\end{align}
where $a_0(\alpha)=1$ and for $i\geq 1$,
\begin{align*}
&a_i(\alpha)=(-1)^i(a+2i)\frac{(a+1)\cdots(a+i-1)}{i!}
\frac{(1-\alpha)\cdots(1-(i-1)\alpha)}{(1+(a+1)\alpha)\cdots(1+(a+i)\alpha)},\\
&b_k(\alpha)=(s+r-2k)!\prod_{1\leq j\leq k}(s+1-j)(\alpha+j).
\end{align*}
\end{proposition}

\begin{proof}
For the statement about $Q_{\lambda^0}(\alpha)$, we have
$e'_{\lambda^0}(\alpha)-e'_{\lambda^i}(\alpha)=-i(1+(a+i)\alpha)$.
Let's consider $\la=\la^0$ in Theorem \ref{T:iteration}, we find
\begin{align*}
a_i(\alpha)=\frac{a+2i}{-i(1+(a+i)\alpha)}\sum_{j=0}^{i-1}a_j(\alpha),
\end{align*}
where we use the fact that $r_{\la^j\la^i}=a+2i$ for $j<i$.

To finish the proof, we prove that the assignment of
$a_i(\al)$ in the statement satisfies the following equality:
\begin{align}\label{aialpha}
\sum_{j=0}^{i-1}a_j(\alpha)=a_i(\alpha)\frac{-i(1+(a+i)\alpha)}{a+2i}.
\end{align}
In fact, for $i=1$, it is immediate. If (\ref{aialpha}) is true, then we have
\begin{align*}
\sum_{j=0}^{i}a_j(\alpha)&=a_i(\alpha)\frac{-i(1+(a+i)\alpha)}{a+2i}+a_i(\alpha)\\
&=a_i(\alpha)\Big(\frac{-i(1+(a+i)\alpha)}{a+2i}+1\Big)\\
&=a_i(\alpha)\frac{(a+i)(1-i\alpha)}{a+2i}\\
&=a_{i+1}(\alpha)\frac{-(i+1)(1+(a+i+1)\alpha)}{a+2(i+1)}.\\
\end{align*}
For the statement about $J_{(\lambda^0)'}$, the proof is exactly the
same and is omitted.
\end{proof}

In \cite{S}, Stanley conjectured that the Littlewood-Richardson coefficient $C_{\mu\nu}^\la(\al)=\langle
J_\mu(\alpha) J_\nu(\alpha),J_\lambda(\alpha)\rangle$ is a
polynomial of $\alpha$ with nonnegative integer coefficients. Except
a few special case (for example $\mu$ is a one row partition),
this conjecture is believed to be open. In the following, we will
give a combinatorial formula for the coefficients.


First, we will give a combinatorial formula for the Jack
symmetric functions based on the iteration formula in Theorem
\ref{T:iteration}. To do this, we need the following definition.
\begin{definition}
For a sequence of partitions
$\delta=(\lambda^0,\lambda^1,\cdots,\lambda^s)$, we say that
$\delta$ is a moving up filtration of partitions starting from $\lambda^0$ and
ending at $\lambda^s$ if $\lambda^i\in M^*(\lambda^{i-1})$ for
$i=1,2,\cdots,s$. For such a filtration we denote its initial partition as $st(\delta)=\lambda^0$ and
the last partition as $ed(\delta)=\lambda^s$. Assume that $\lambda\geq\lambda^s$, and
$\lambda^0\geq\mu$ we define
$$f^\lambda(\delta)=\prod_{i=0}^{s-1}\frac{r_{\lambda^i\lambda^{i+1}}}{e'_\lambda(\alpha)-e'_{\lambda^i}(\alpha)},$$
$$f_\mu(\delta)=\prod_{i=0}^{s-1}\frac{r_{\lambda^i\lambda^{i+1}}}{e'_\mu(\alpha)-e'_{\lambda^{i+1}}(\alpha)}.$$
\end{definition}

\begin{theorem}\label{T:combinatorial formula}
Let $J_\lambda(\alpha)=\sum_{\mu}v_{\lambda\mu}(\alpha)m_\mu$,
$Q_\mu(\alpha)=\sum_{\la}C_{\mu\lambda}(\alpha)q_\lambda(\alpha)$
for $\mu<\lambda$ we have

\begin{align*}
&v_{\lambda\mu}(\alpha)=v_{\lambda\lambda}(\alpha)\sum
f^\lambda(\delta),\\
&C_{\mu\lambda}(\alpha)=C_{\mu\mu}(\alpha)\sum f_\mu(\delta),
\end{align*}
 where both sums are over all moving up filtrations of partitions
$\delta$ from $\mu$ to $\lambda$.
\end{theorem}
Note that $C_{\mu\mu}=1$ and
$v_{\lambda\lambda}(\alpha)=\sum_{s\in\lambda}h_*^\lambda(s)$. Also
notice that $v_{\lambda\mu}(\alpha)$ is an integral polynomial of
$\al$ by Theorem \ref{T:KSMac}, we have:
\begin{corollary} The coefficient
$v_{\lambda\mu}$ is the product of integral polynomials of the form
$a_i\alpha+b_i$ if there is a unique moving up filtration from $\mu$
to $\lambda$.
\end{corollary}

Note that we can also write theorem \ref{T:combinatorial formula} in
another form as
\begin{align*}
&J_\lambda(\alpha)=v_{\la\la}(\al)\sum_{ed(\delta)=\la}f^\la(\delta)m_{st(\delta)},\\
&Q_\mu(\alpha)=C_{\mu\mu}(\alpha)\sum_{st(\delta)=\mu}f_\mu(\delta)q_{ed(\delta)}(\alpha).
\end{align*}
Using this formula, we can give a combinatorial formula
for the Littlewood-Richardson coefficients of Jack functions.

\begin{theorem}
We have\begin{align*}
 \langle Q_\mu(\al)Q_\nu(\al),
J_\lambda(\al)\rangle=v_{\la\la}(\al)\sum_{\delta_1,\delta_2,\delta}f_\mu(\delta_1)f_\nu(\delta_2)f^\la(\delta),
\end{align*}
where the sum is over all triples of moving up filtrations
$(\delta_1,\delta_2,\delta)$ such that
$st(\delta_1)=\mu,~st(\delta_2)=\nu,~ed(\delta)=\la$, and
 $st(\delta)=ed(\delta_1)\cup ed(\delta_2)$.
\end{theorem}

We have the following rising-operator-like formula for Jack
functions as a corollary of Theorem \ref{T:combinatorial formula}.
\begin{corollary}\label{C:rising-operator-like formula}
We have, for any partition $\la$,
\begin{equation}
 Q_\la(\al)=\sum_{(\underline{r})}\prod_{t=1}^l\frac{(r_{t-1}\cdots
r_2r_1.\la)_{i_t}-(r_{t-1}\cdots
r_2r_1.\la)_{j_t}+2k_t}{e'_\la(\al)-e'_{r_t\cdots
r_1.\la}(\al)}q_{r_l\cdots r_2r_1.\la}(\al),
\end{equation}
where the sum is over all sequences $(\underline{r})=(r_l,\cdots,r_2,r_1)$, here $r_p$ denotes the moving up operator
$r_{j_p}^{i_p}(k_p)$ (see Definition
\ref{D:risingfiltration}), and
$e_{\lambda}'(\alpha)$ is given in (\ref{eq:eprime}). When
$l=0$ it corresponds to the term $q_\la(\al)$.
\end{corollary}
Note that $q_{r_j^i(k).\la}=R_{ij}^k.q_\la$, using the usual rising
operator as in Corollary \ref{C:action-of-D} (Thus the summands
corresponding to $l\leq 1$ are essentially given in rising operator
formula). Like the usual rising operator formula, for each
$\mu\geq\la$, only finitely many sequences $(r_l,\cdots,r_2,r_1)$
contribute to the term $q_\mu$. In this sense our raising operator formula generalizes
the canonical Schur case to the Jack case. The difference from a usual raising
operator formula is that one needs to rearrange the parts before the next action.

\bibliographystyle{amsalpha}

\begin{thebibliography}{ABCD}

\bibitem[AMOS]{AMOS} H. Awata, Y. Matsuo, S. Odake, J. Shiraishi, {\em
Collective field theory, Calogero-Sutherland model, and generalized
matrix models}, Phys. Lett. B 347 (1995), 49-55.

\bibitem[CJ]{CJ} W. Cai, N. Jing, {\em On vertex operator realizations of Jack functions},
 Jour. Alg. Comb. 32 (2010), 579-595.

\bibitem[EJ]{EJ} \"O. E\~gecio\~glu, J. Remmel, {\em Brick tabloids and the connection matrices between
bases of symmetric functions}, Discrete. Appl. Math. 34 (1991),
107-120.

\bibitem[FF]{FF} B. L. Feigin, D. B. Fuchs, {\em Verma modules over the Virasoro algebra},
Topology (Leningrad, 1982), 230--245, Lecture Notes in Math., 1060,
Springer, Berlin, 1984.

\bibitem[FW]{FW} I. Frenkel, W. Wang, {\em Virasoro algebra and wreath product convolution},
 J. Algebra 242 (2001), no. 2, 656--671.

\bibitem[G]{G} A. Garsia, {\em Orthogonality of Milne's polynomials and raising operators},
 Discrete Math. 99 (1992), no. 1-3, 247--264.

\bibitem [HHL] {HHL} J. Haglund, M. Haiman, N. Loehr, {\em A
combinarorial formula for Macdonald Polynomials}, J. Amer. Math.
Soc. 18 (2005), 735-761.

\bibitem [I]{I} S. Iso, {\em Anyon basis of $c=1$ conformal field
theory}, Nucl. Phys. B 443 [FS] (1995), 581-595.

\bibitem[Ja]{Ja} H. Jack, {\em A class of symmetric polynomials with a
parameter}, Proc. Roy. Soc. Edinburgh Sect. A 69 (1970/1971) 1-18.

\bibitem[J1]{J1} N. Jing, {\em Vertex operators and Hall-Littlewood symmetric functions},
Adv. Math. 87 (1991), no. 2, 226--248.

\bibitem[J2]{J2} N. Jing, {\em $q$-hypergeometric series and Macdonald
functions}, Jour. Alg. Comb. 3 (1994) 291-305.

\bibitem[JJ]{JJ} N. Jing, T. J\'ozefiak, {\em A formula for two row Macdonald
functions}, Duke Math. J. 67, No. 2 (1992), 377-385.

\bibitem [KS]{KS} F. Knop, S. Sahi, {\em A recursion and a combinatorial formula for
Jack polynomials}, Invent. Math. 128 (1997), 9-22.

\bibitem[LLM] {LLM} L. Lapointe, A. Lascoux, J. Morse {\em Determinantal expression and recursion for Jack
polynomials}, Electron. J. Combin. 7 (2000), Note 1, 7 pp.
(electronic)

\bibitem[M]{M} I. G. Macdonald, {\em Symmetric functions and Hall polynomials},
 2nd ed., With contributions by A. Zelevinsky. Oxford Univ. Press, New York, 1995.

\bibitem[M1]{M1} I. G. Macdonald, {\em
Commuting differential operators and zonal spherical functions}.
Algebraic groups Utrecht 1986, 189--200, Lecture Notes in Math.,
1271, Springer, Berlin, 1987.

\bibitem[MY]{MY} K. Mimachi, Y. Yamada, {\em Singular vectors of the Virasoro
algebra in terms of Jack symmetric polynomials}, Comm. Math. Phy.,
174 (1995), no. 2, 447-455.

\bibitem[S]{S} R. Stanley, {\em Some combinatorial properties of Jack symmetric
functions}, Adv. Math. 77 (1989), no. 1, 76--115.

\bibitem[So]{So} K. Sogo, {\em Eigenstates of Calogero-Sutherland-Moser
model and generalized Schur functions}, J. Math. Phys. 35
(1994),2282-2296.

\bibitem[SAFR]{SSAFR} R. Sakamoto, J. Shiraishi, D. Arnaudon, L.
Frappat, E. Ragoucy, {\em Correspondence between conformal firld
theothy and Calogero-Sutherland model}, Nucl. Phys. B 704 (2005),
490-509.

\end{thebibliography}

\end{document}